\documentclass[reqno, amsthm, 11pt, a4paper]{amsart}
\usepackage{clrscode}
\usepackage{palatino}
\usepackage{mathpazo}
\usepackage{tikz}
\usepackage[small]{caption}
\usetikzlibrary{shapes,arrows}
\usepackage{mathrsfs,amssymb}

\usepackage{amsfonts}
\usepackage{amsmath}
\usepackage{amssymb}
\usepackage{latexsym}
\usepackage{mathrsfs}
\usepackage{graphicx}

\usepackage{epic,eepic,ecltree} 

\input xy
\xyoption{all}

\usepackage{xy} 

\newcommand{\dedge}[1]{\ar@{--}[#1]}
\newcommand{\edge}[1]{\ar@{-}[#1]}
\newcommand{\lulab}[1]{\ar@{}[l]_<<{#1}}
\newcommand{\rulab}[1]{\ar@{}[r]^<<{#1}}
\newcommand{\ldlab}[1]{\ar@{}[l]^<<{#1}}
\newcommand{\rdlab}[1]{\ar@{}[r]_<<{#1}}

 \newcommand{\gh}{\mathscr{H}}
  \newcommand{\gr}{\mathscr{R}}
  \newcommand{\gl}{\mathscr{L}}
  \newcommand{\gp}{\mathscr{P}}

  \newcommand{\gG}{\mathscr{G}}

 \newcommand{\ZG}{{\mathbb{Z}G}}

 \newcommand{\ZM}{{\mathbb{Z}M}}

 \newcommand{\FP}{{\rm FP}}
 \newcommand{\FPinfty}{{\rm FP}\sb \infty}

 \newcommand{\bbe}{\mathbb{E}}
 \newcommand{\bbp}{\mathbb{P}}
 \newcommand{\bbq}{\mathbb{Q}}
 \newcommand{\bbr}{\mathbb{R}}
 \newcommand{\bbs}{\mathbb{S}}

\newtheorem{thm}{Theorem}

\newtheorem{cor}{Corollary}

\newtheorem{prop}{Proposition}
\newtheorem{lem}{Lemma}

\theoremstyle{definition}

\newcommand{\lb}{\langle}
\newcommand{\rb}{\rangle}
\usepackage[active]{srcltx}

\usepackage{verbatim}
\usepackage{enumerate}

\begin{document}

\title[]{
On properties not inherited by monoids from \\ their Sch\"{u}tzenberger groups}

\maketitle

\vspace{-6mm}

\begin{center}

    R. GRAY\footnote{
	  Part of this work was done while this author held an EPSRC Postdoctoral Fellowship at the University of St Andrews.}

    \medskip

    Centro de \'{A}lgebra da Universidade de Lisboa, \\ Av. Prof. Gama Pinto, 2,  1649-003 Lisboa,  \ Portugal.

    \medskip

    \texttt{rdgray@fc.ul.pt}

    \bigskip

    A. MALHEIRO\footnote{
	This work was developed within the projects POCTI-ISFL-1-143 and PTDC/MAT/69514/2006 of CAUL, financed by FCT and FEDER, and
	Supported by the Treaty of Windsor scheme of the British Council of Great Britain and Portugal. }

    \medskip

    Centro de \'{A}lgebra da Universidade de Lisboa, \\ Av. Prof. Gama Pinto, 2,  1649-003 Lisboa,  \ Portugal.
    \\ and \\ Departamento de Matem\'{a}tica, Faculdade de Ci\^{e}ncias e Tecnologia \\ da Universidade Nova de Lisboa,
    2829-516 Caparica, Portugal

    \medskip

    \texttt{malheiro@cii.fc.ul.pt}

    \bigskip

    S. J. PRIDE

    \medskip

    Department of Mathematics,\ University of Glasgow,  \\
    University Gardens, G12 8QW, Scotland.

    \medskip

    \texttt{stephen.pride@gla.ac.uk} \\
\end{center}

\begin{abstract}
We give an example of a monoid with finitely many left and right ideals, all of whose Sch\"{u}tzenberger groups are presentable by finite complete rewriting systems, and so each have finite derivation type, but such that the monoid itself does not have finite derivation type, and therefore does not admit a presentation by a finite complete rewriting system. The example also serves as a counterexample to several other natural questions regarding complete rewriting systems and finite derivation type. Specifically it allows us to construct two finitely generated monoids $M$ and $N$ with isometric Cayley graphs, where $N$ has finite derivation type (respectively, admits a presentation by a finite complete rewriting system) but $M$ does not. This contrasts with the case of finitely generated groups for which finite derivation type is known to be a quasi-isometry invariant. The same example is also used to show that neither of these two properties is preserved under finite Green index extensions.

\

\end{abstract}

\section{Introduction}
\label{sec_intro}

It is well known that, even if a monoid is given by a finite presentation, the word problem for the monoid may be undecidable. Markov and Post proved independently that the word problem for finitely presented monoids is undecidable in general. Later, Novikov and Boone extended the result of Markov and Post to finitely presented groups; see \cite{LyndonandSchupp} for references. Therefore, one is interested in classes of finite presentations which guarantee that the word problem is decidable. A class of this form that has received a lot of attention in the literature is the class of presentations that are finite and complete (also called convergent). A finite complete rewriting system is a finite presentation for a monoid of a particular form (both confluent and noetherian) which in particular gives a solution of the word problem for the monoid.
(See Section~\ref{sec_prelims} for full definitions of the concepts mentioned here.)
It is natural to seek an algebraic characterization of the class of finitely presented monoids that admit a presentation through a finite complete
string rewriting system.   As part of this investigation, in \cite{Squier1} Squier introduced a homotopical finiteness property of monoids called finite derivation type. Given a rewriting system (i.e. monoid presentation) $\lb A | R \rb$ one builds a (combinatorial) $2$-complex $\mathcal{D}$, called the \emph{Squier complex}, whose $1$-skeleton has vertex set $A^*$ and edges corresponding to applications of relations from $R$, and that has $2$-cells adjoined for each instance of `non-overlapping' applications of relations from $R$. There is a natural action of the free monoid $A^*$ on the Squier complex $\mathcal{D}$. A  collection of closed paths in $\mathcal{D}$ is called a \emph{homotopy base} if the complex obtained by adjoining cells for each of these paths, and those that they generate under the action of the free monoid on the Squier complex, has trivial fundamental groups. A monoid defined by a presentation is said to have \emph{finite derivation type} (FDT for short) if the corresponding Squier complex admits a finite homotopy base. Squier \cite{Squier1} proved that the property FDT is independent of the choice of finite presentation, so we may speak of FDT monoids. The original motivation for studying this notion is Squier's result \cite{Squier1} which says that if a monoid admits a presentation by a finite complete rewriting system then the monoid must have FDT. Further motivation for the study of these concepts comes from the fact that the fundamental groups of connected components of Squier complexes, which are called \emph{diagram groups}, have turned out to be a very interesting class of groups, and have been extensively studied in
\cite{Guba2006, Farley2003, Farley2005, Guba, Guba1999, GubaSapir2006}.
Various other geometric finiteness properties have been introduced and investigated in the study of complete rewriting systems; see for instance 
\cite{Alonso2003, KobayashiOtto2003}.
String rewriting systems continue to receive a lot of attention in the literature; see \cite{Brown2006, Chouraqui2009, Evans2007, Hermiller1999}.
More background on the connections between string rewriting systems and homological and homotopical finiteness properties of monoids may be found in the survey articles \cite{CohenSurvey, OttoKobayashiSurvey}.

It is natural to seek connections between the properties of a monoid and the properties of the subgroups of that monoid, and numerous results of this kind exist in the literature. For instance, in \cite{Ruskuc1999} it was shown that a (von Neumann) regular monoid $S$ with finitely many left and right ideals is finitely presented if and only if all of its maximal subgroups are finitely presented. Analogous results are known to hold for numerous other finiteness properties, and this result remains valid if finite presentability is replaced by various other standard
finiteness conditions. In particular we have (under the same assumptions on $S$) that  $S$ is residually finite (respectively,
locally finite, periodic, finitely generated, with solvable word problem) if and only if all the
maximal subgroups of $S$ are residually finite (respectively, locally finite, periodic, finitely generated,
with solvable word problem); see \cite{Golubov1975, Ruskuc1999}.
More recently in \cite[Theorem~10.12]{Dales2010} it was shown how the amenability of the Banach algebra associated with a semigroup relates to the amenability of the maximal subgroups of the semigroup.
It follows from a result in \cite{Duncan1990} that if the Banach algebra associated with the semigroup is amenable then the semigroup must be regular with finitely many left and right ideals. Regular semigroups with finitely many left and right ideals also arise naturally in the study of 
free regular idempotent generated semigroups of finite biordered sets; see for instance \cite{Brittenham2009, Nambooripad1979, Gray2011}.

It was pointed out in \cite[Remark and Open Problem 4.5]{Ruskuc1999} that 
the situation was less clear for various finiteness conditions related to homology and rewriting systems, and specifically it was asked whether corresponding results to those mentioned in the previous paragraph hold for either the property of being presentable by a finite complete rewriting system, or the related homotopical finiteness property $\mathrm{FDT}$.
In recent work by the first two authors of the present article, in \cite[Theorem~2]{GrayMalheiro} it was proved that a regular monoid $S$ with finitely many left and right ideals has $\mathrm{FDT}$ if and only if all its maximal subgroups have $\mathrm{FDT}$, while in \cite[Theorem~1]{GrayMalheiro2} it is shown that a regular monoid with finitely many left and right ideals is presented by a finite complete rewriting system provided all of its maximal subgroups admit presentations by finite complete rewriting systems.

Given these results, one natural thing to do is to try and extend them from regular monoids to arbitrary (non-regular) monoids. 
Recall that the maximal subgroups of a monoid $S$ are precisely the $\gh$-classes (in the sense of \cite{Green1}) of $S$ that contain idempotents. Sch\"{u}tzenberger
\cite{Schutzenberger1957, Schutzenberger1958} showed how one can assign to an arbitrary $\gh$-class $H$ a group
$\mathcal{G}(H)$, called the Sch\"{u}tzenberger group of $H$. Sch\"{u}tzenberger groups have
many features in common with maximal subgroups. In particular, if the $\gh$-class
$H$ contains an idempotent (and hence is a maximal subgroup) then $H$
and $\mathcal{G}(H)$ are isomorphic. Generalising the above-mentioned result for regular semigroups to arbitrary semigroups,
the main result of
\cite{Ruskuc2000} asserts that a monoid with finitely many left and right ideals is finitely presented if and only if all its Sch\"{u}tzenberger groups are finitely presented. As in the regular case, results like this are not particular to finite presentability, and the same result holds with finite presentability replaced by a long list of standard finiteness properties including being: residually finite, locally finite, periodic, finitely generated, or having solvable word problem; see \cite{Golubov1975, Gray&Ruskuc4, Ruskuc2000}. From this evidence, it would not be unreasonable to suppose that the results about finite complete rewriting systems, and $\mathrm{FDT}$, for regular monoids mentioned in the previous paragraph (obtained in \cite{GrayMalheiro2, GrayMalheiro}) should, as for all the other properties mentioned above, extend to non-regular semigroups via the concept of Sch\"{u}tzenberger group. The aim of this article is to show that, in fact, contrary to expectation, this is not the case. We do this by giving an example of a monoid with finitely many left and right ideals, all of whose Sch\"{u}tzenberger groups are given by finite complete rewriting systems, and therefore all have $\mathrm{FDT}$, but such that the monoid itself does not have $\mathrm{FDT}$, and therefore does not admit a presentation by a finite complete rewriting system. This is the main result of this article.

\begin{thm}
\label{thm_themainresult}
Let $M$ be the monoid defined by the presentation
\begin{align*}
\lb \; a, a^{-1}, b, b^{-1}, h
& \; | \; aa^{-1} = a^{-1}a=bb^{-1}=b^{-1}b=1, \\
& \phantom{\; | \;}
xh = hx, \; hxy = hyx \quad (x,y \in \{ a, a^{-1}, b, b^{-1} \}),  \\
& \phantom{\; | \;}
h^2 a = h^2 a^{-1} = h^2 b = h^2 b^{-1} = h^3 = h^2 \; \rb.
\end{align*}
\begin{enumerate}[{\normalfont (i)}]
\item The monoid $M$ has exactly three $\gh$-classes, with Sch\"{u}tzenberger groups isomorphic to the trivial group, the free abelian group of rank two, and the free group of rank two, respectively. In particular, all three Sch\"{u}tzenberger groups of $M$ are presentable by finite complete rewriting systems and they all have finite derivation type.
\item The monoid $M$ does not have finite derivation type, and therefore is not presentable by a finite complete rewriting system.
\end{enumerate}
\end{thm}
Part (i) of this theorem is very straightforward to verify, see Section~\ref{sec_up} below. Far less obvious is part (ii), and most of Section~\ref{sec_up} will be devoted to its proof. 
Let us make a few further remarks about this result.  

\

\noindent $\bullet$ Although the monoid $M$ does not have $\mathrm{FDT}$, it does have numerous other desirable properties. In particular we shall see that $M$ is of type left- and right-$\mathrm{FP}_\infty$, and $M$ has a linear time solvable word problem. 

\

\noindent $\bullet$  Using this example, exploiting the way the example highlights the different way that the properties behave for non-regular monoids when compared to regular monoids, we shall show (in Section~\ref{sec_consequences}) that neither the property of being presented by a finite complete rewriting system, nor $\mathrm{FDT}$, are isometry invariants of monoids. 
That is, we give examples of two finitely presented monoids $M$ and $N$, and generating sets, such that the resulting pair of Cayley graphs are isometric as directed spaces, but where $N$ has $\mathrm{FDT}$ (and is defined by a finite complete rewriting system) while $M$ does not have $\mathrm{FDT}$ (and is therefore not definable by a finite complete rewriting system). This contrasts with the case of groups for which $\mathrm{FDT}$ (which is equivalent to $\FP_3$ for groups \cite{Cremanns1996}) is known by \cite{Alonso1994} to be a quasi-isometry invariant.

\

\noindent $\bullet$  We shall also use Theorem~\ref{thm_themainresult} to show that neither $\mathrm{FDT}$, nor the property of being definable by a finite complete rewriting system, is inherited by finite Green index extensions, in the sense of \cite{Cain2009,Gray2008}. This is a little surprising, since both of these finiteness properties are known to be preserved when taking finite index extensions of groups
(see \cite[Proposition 5.1]{Brown} and \cite{Groves1993}), and when taking finite extensions of semigroups (see \cite{Wang1998}).

\

\noindent $\bullet$  For $\mathrm{FDT}$, in the other direction, passing from the monoid to its Sch\"{u}tzenberger groups, the expected result does hold: as a corollary of the main result of \cite{GrayMalheiroPride1} we have that if $S$ is a monoid with finitely many left and right ideals, and $S$ has $\mathrm{FDT}$, then all Sch\"{u}tzenberger groups of $S$ have $\mathrm{FDT}$.

\

\noindent In addition to this introduction, this article comprises four sections. 
In Section~2 we recall some basic definitions and results about string rewriting systems, and 
finite derivation type, and give the necessary notions from the structure theory of semigroups that we shall need. 
The proof of our main result, Theorem~\ref{thm_themainresult}, is given in Section~3. 
Finally, in Section~4 we discuss some consequences of our main result.

\section{Preliminaries}
\label{sec_prelims}

\subsection*{Derivation graphs, homotopy bases, and finite derivation type}

Let $A$ be a finite alphabet and let $R$ be a (possibly infinite) rewriting system over $A$. That is, $R \subseteq A^* \times A^*$ where $A^*$ is the free monoid over $A$. We assume, without loss of generality, that $R$ is anti-symmetric (that is, that $(u,v) \in R$ implies $(v,u) \not\in R$). An element of $R$ is called a \emph{rule}, and we often write $r_{+1}=r_{-1}$ for $(r_{+1}, r_{-1}) \in R$. For $u, v \in A^*$ we write $u \rightarrow_R v$ if $u \equiv w_1 r_{+1} w_2$, and $v \equiv w_1 r_{-1} w_2$ where $(r_{+1},r_{-1}) \in R$ and $w_1, w_2 \in A^*$. 
Here we write $u \equiv w$, for words $u,w \in A^*$, to mean that $u$ and $w$ are equal as words in $A^*$. 
The reflexive symmetric transitive closure $\leftrightarrow_R^{*}$ of $\rightarrow_R$ is precisely the congruence on $A^*$ generated by $R$. The ordered pair $\langle A | R \rangle$ is called a \emph{monoid presentation} with \emph{generators} $A$ and set of \emph{defining relations} $R$. If $S$ is a monoid that is isomorphic to $A^*  / \leftrightarrow_R^{*}$ we say that $S$ is the \emph{monoid defined by the presentation $\lb A | R \rb$}.
We say that two rewriting systems over the same alphabet are (Tietze) \emph{equivalent} if they define the same monoid. 
We write $|w|$ to denote the total number of letters in a word $w \in A^*$, which we call the \emph{length} of the word $w$.

With any monoid presentation $\gp = \lb A | R \rb$ we associate a graph (in the sense of Serre \cite{SerreTrees}) as follows. The \emph{derivation graph} of $\gp = \lb A | R \rb$ is an infinite graph $\Gamma = \Gamma(\gp) = (V,E,\iota, \tau, ^{-1})$ with \emph{vertex set} $V = A^*$, and \emph{edge set $E$} consisting of the collection of $4$-tuples
\[
\{ (w_1, r, \epsilon, w_2): \ w_1, w_2 \in A^*, r \in R, \ \mbox{and} \ \epsilon \in \{ +1, -1 \} \}.
\]
The functions $\iota, \tau : E \rightarrow V$ associate with each edge $\bbe = (w_1, r, \epsilon, w_2)$ (with $r=(r_{+1},r_{-1}) \in R$) its initial and terminal vertices $\iota \bbe = w_1 r_{\epsilon} w_2$ and $\tau \bbe = w_1 r_{- \epsilon} w_2$, respectively. The mapping $^{-1} : E \rightarrow E$ associates with each edge $\bbe = (w_1, r, \epsilon, w_2)$ an inverse edge $\bbe^{-1} = (w_1, r, -\epsilon, w_2)$.

A path is a sequence of edges $\bbp = \bbe_1 \circ \bbe_2 \circ \ldots \circ \bbe_n$ where $\tau \bbe_i  \equiv \iota \bbe_{i+1}$ for $i=1, \ldots, {n-1}$. Here $\bbp$ is a path from $\iota \bbe_1$ to $\tau \bbe_n$ and we extend the mappings $\iota$ and $\tau$ to paths by defining $\iota \bbp \equiv \iota \bbe_1$ and $\tau \bbp \equiv \tau \bbe_n$.  The inverse of a path $\bbp = \bbe_1 \circ \bbe_2 \circ \ldots \circ \bbe_n$ is the path $\bbp^{-1} = \bbe_n^{-1} \circ \bbe_{n-1}^{-1} \circ \ldots \circ \bbe_1^{-1}$, which is a path from $\tau \bbp$ to $\iota \bbp$. A \emph{closed path} is a path $\bbp$ satisfying $\iota \bbp \equiv \tau \bbp$. For two paths $\bbp$ and $\bbq$ with $\tau \bbp \equiv \iota \bbq$ the composition $\bbp \circ \bbq$ is defined.

We denote the set of paths in $\Gamma$ by $P(\Gamma)$, where for each vertex $w \in V$ we include a path $1_w$ with no edges, called the \emph{empty path} at $w$. We call a path $\bbp$ \emph{positive} if it is either empty or it contains only edges of the form $(w_1, r, +1, w_2)$. We use $P_+(\Gamma)$ to denote the set of all positive paths in $\Gamma$. Dually we have the notion of $\emph{negative}$ path, and $P_-(\Gamma)$ denotes the set of all negative paths. The free monoid $A^*$ acts on both sides of the set of edges $E$ of $\Gamma$ by
\[
x \cdot \bbe \cdot y = (x w_1, r, \epsilon, w_2 y)
\]
where $\bbe = (w_1, r, \epsilon, w_2)$ and $x, y \in A^*$. This extends naturally to a two-sided action of $A^*$ on $P(\Gamma)$ where for a path $\bbp = \bbe_1 \circ \bbe_2 \circ \ldots \circ \bbe_n$ we define
\[
x \cdot \bbp \cdot y = (x \cdot \bbe_1 \cdot y) \circ (x \cdot \bbe_2 \cdot y)
\circ \ldots \circ (x \cdot \bbe_n \cdot y).
\]
If $\bbp$ and $\bbq$ are paths such that $\iota \bbp \equiv \iota \bbq$ and $\tau \bbp \equiv \tau \bbq$ then we say that $\bbp$ and $\bbq$ are \emph{parallel}, and write $\bbp \parallel \bbq$. We use $\parallel\; \subseteq P(\Gamma) \times P(\Gamma)$ to denote the set of all pairs of parallel paths.

An equivalence relation $\sim$ on $P(\Gamma)$ is called a \emph{homotopy relation} if it is contained in $\parallel$ and satisfies the following four conditions.

\begin{enumerate}[(H1)]

\item If $\bbe_1$ and $\bbe_2$ are edges of $\Gamma$, then
\[
(\bbe_1 \cdot \iota \bbe_2) \circ (\tau \bbe_1 \cdot \bbe_2) \sim
(\iota \bbe_1 \cdot \bbe_2) \circ (\bbe_1 \cdot \tau \bbe_2 ).
\]

\item For any $\bbp, \bbq \in P(\Gamma)$ and $x,y \in A^*$
\[
\bbp \sim \bbq \ \ \mbox{implies} \ \  x \cdot \bbp \cdot y \sim x \cdot \bbq \cdot y.
\]

\item For any $\bbp, \bbq, \bbr, \bbs \in P(\Gamma)$ with $\tau \bbr \equiv \iota \bbp \equiv \iota \bbq$ and $\iota \bbs \equiv \tau \bbp \equiv \tau \bbq$
\[
\bbp \sim \bbq \ \ \mbox{implies} \ \ \bbr \circ \bbp \circ \bbs \sim \bbr \circ \bbq \circ \bbs.
\]

\item If $\bbp \in P(\Gamma)$ then $\bbp \bbp^{-1} \sim 1_{\iota \bbp}$, where $1_{\iota \bbp}$ denotes the empty path at the vertex $\iota \bbp$.

\end{enumerate}

For a subset $C$ of $\parallel$, the homotopy relation \emph{$\sim_C$ generated by $C$} is the smallest (with respect to inclusion) homotopy relation containing $C$. The relation $\sim_0 = \sim_\varnothing$ generated by the empty set $\varnothing$ is the smallest homotopy relation. If $\sim_C$ coincides with $\parallel$, then $C$ is called a \emph{homotopy base} for $\Gamma$. The presentation $\lb A | R \rb$ is said to have \emph{finite derivation type} ($\mathrm{FDT}$) if the derivation graph $\Gamma$ of $\lb A | R \rb$ admits a finite homotopy base.
A finitely presented monoid $S$ is said to have \emph{finite derivation type} ($\mathrm{FDT}$) if some (and hence any by \cite[Theorem~4.3]{Squier1})  finite presentation for $S$ has finite derivation type.

It is not difficult to see that a subset $C$  of $\parallel$ is a homotopy base of $\Gamma$ if and only if the set \[\{(\bbp \circ \bbq^{-1}, 1_{\iota \bbp}): (\bbp,\bbq) \in C \}\] is a homotopy base for $\Gamma$. Thus we say that a set $D$ of circuits is a homotopy base if the corresponding set $\{ (\bbp, 1_{\iota \bbp}) : \bbp \in D \}$ is a homotopy base. The following easy lemma will prove useful.

\begin{lem}[\cite{Kobayashi2000}, Lemma~2.1]
\label{lem_homgen}
A set $C$ of circuits in $\Gamma = \Gamma(\gp)$ is a homotopy base if and only if for any circuit $\bbp$ in $\Gamma$, there are $v_i, w_i \in A^*$, $\bbp_i \in P(\Gamma)$ and $\bbq_i \in C \cup C^{-1}$, $i=1,\ldots,n$, $n \geq 0$, such that
\[
\bbp
\sim_0
\bbp_1^{-1} \circ (v_1 \cdot \bbq_1 \cdot w_1) \circ \bbp_1 \circ \cdots \circ \bbp_n^{-1} \circ (v_n \cdot \bbq_n \cdot w_n) \circ \bbp_n.
\]
\end{lem}
Let us conclude this subsection by describing a standard method for obtaining a (possibly infinite) homotopy base for a presentation.
Let $\alpha_1 \cdot \bbe_1 \cdot \beta_1$ and $\alpha_2 \cdot \bbe_2 \cdot \beta_2$ be two edges of a derivation graph $\Gamma(\gp)$ such that $\alpha \equiv \alpha_1 u_1 \beta_1 \equiv \alpha_2 u_2 \beta_2$, where
\[
\bbe_1 = (1, u_1 = v_1, +1, 1), \quad 
\bbe_2 = (1, u_2 = v_2, +1, 1).
\]
We call a path \[\bbp = (\alpha_1 \cdot \bbe_1^{-1} \cdot \beta_1) \circ (\alpha_2 \cdot \bbe_2 \cdot \beta_2)\] a \emph{peak}. If $u_1$ and $u_2$ do not overlap in $\alpha$ (that is, if $|\alpha_1 u_1| \leq |\alpha_2|$ or $|\alpha_1| \geq |\alpha_2 u_2|$) then $\bbp$ is called a \emph{disjoint peak}.

If the peak $\bbp = (\alpha_1 \cdot \bbe_1^{-1} \cdot \beta_1) \circ (\alpha_2 \cdot \bbe_2 \cdot \beta_2)$ is not disjoint then, up to symmetry, the situation breaks down into the following two cases:
\begin{enumerate}
\item[(i)] $u_1$ is a factor of $u_2$, that is, $u_2$ can be written as $u_2 \equiv \gamma_1 u_1 \gamma_2$ for some $\gamma_1, \gamma_2 \in A^*$, or
\item[(ii)] $u_1$ overlaps with $u_2$ on the left, that is, $u_1 \gamma_1 \equiv \gamma_2 u_2$ for some $\gamma_1, \gamma_2 \in A^+$ satisfying $|\gamma_2| < |u_1|$.
\end{enumerate}
In case (i) we have \[\bbp = \alpha_2 \cdot ( (\gamma_1 \cdot \bbe_1^{-1} \cdot \gamma_2) \circ \bbe_2) \cdot \beta_2,\] while in case (ii) we have \[\bbp = \alpha_1 \cdot ( (\bbe_1^{-1} \cdot \gamma_1) \circ (\gamma_2 \cdot \bbe_2)) \cdot \beta_2.\] The paths $(\gamma_1 \cdot \bbe_1^{-1} \cdot \gamma_2) \circ \bbe_2$ and $(\bbe_1^{-1} \cdot \gamma_1) \circ (\gamma_2 \cdot \bbe_2)$ are called \emph{critical peaks}. A critical peak $\bbq$ is \emph{resolvable} is there exists $w \in A^*$ and positive paths $\bbp_1$ from $\iota \bbq$ to $w$, and $\bbp_2$ from $\tau \bbq$ to $w$, in which case we call $\bbp_1^{-1} \circ \bbq \circ \bbp_2$ a \emph{critical circuit}.

Recall that a rewriting system $R$ is called \emph{complete} (or \emph{convergent}) if it is noetherian and confluent. This means that $R$ does not admit any infinite reduction sequences, and whenever $w$ reduces to two strings $u$ and $v$, then $u$ and $v$ have a common descendant in the system. It is well known (see for instance \cite{Book1}) that a noetherian system is confluent if and only if every critical peak is resolvable.

The following lemma is the essential part of Squier's theorem \cite{Squier1} stating that a monoid defined by a finite complete rewriting system has FDT.

\begin{lem}
\label{lem_criticalcircuits}
If $\gp = \lb A | R \rb$ is a complete rewriting system, then the set of critical circuits forms a homotopy base for $\Gamma(\gp)$.
\end{lem}

Note that the above lemma applies even in the case that the rewriting system is infinite.

For any rewriting system we can find an infinite homotopy base using the following general approach.
Let $\gp = \lb A | R \rb$ be monoid presentation and let $\overline{R}$ be a complete rewriting system that contains $R$ and it is equivalent to $R$, meaning that $\lb A | R \rb$ and $\lb A | \overline{R} \rb$ are Tietze equivalent. Such a system always exist by standard results; see \cite{Book1}. Let $\Gamma = \Gamma(\gp)$ and $\overline{\Gamma} = \Gamma(\lb A | \overline{R} \rb)$ be the corresponding derivation graphs. For each edge $\bbe$ of $\overline{\Gamma}$ choose a path $\bbp_\bbe$ in $\Gamma$ as follows. If $\bbe \in \Gamma$ then take $\bbp_\bbe = \bbe$, otherwise fix some path in $\Gamma$ that leads from $\iota \bbe$ to $\tau \bbe$. Such a path exists since $R$ and $\overline{R}$ are equivalent. Then let $\varphi: P(\overline{\Gamma}) \rightarrow P(\Gamma)$ be the map extending $\bbe \mapsto \bbp_\bbe$ in the obvious natural way. Then we have the following. 
\begin{lem}\cite[Lemma~2.4]{Kobayashi5}
\label{lem_infinitetofinite}
If $C$ is a set of resolutions of all critical peaks of $\overline{\Gamma}$ then $\varphi(C)$ is a homotopy base for $\Gamma$.
\end{lem}

\subsection*{Green's relations and Sch\"{u}tzenberger groups}

The rest of this section is spent recalling some fundamental ideas from the structure theory of semigroups.
For more details about Green's relations, and other basic notions from semigroup theory, we refer the reader to \cite{Howie1}, or more recently \cite{SteinbergBook2009}.

One obtains significant information about a semigroup by considering its ideal structure. Since their introduction in \cite{Green1}, Green's relations have become a fundamental tool for describing the ideal structure of semigroups.
If $S$ is a monoid then Green's relations $\gr$, $\gl$ and $\gh$ are defined by $a \gr b$ if and only if $aS = bS$, $a \gl b$ if and only if $Sa = Sb$, and $\gh = \gr \cap \gl$.
Clearly each of $\gr$, $\gl$ and $\gh$ is an equivalence relation on $S$.
The importance of the $\gh$ relation becomes clear when one begins investigating the subgroups of a monoid (that is, those subsemigroups which form  groups under the semigroup operation). If $H$ is an $\gh$-class containing an idempotent $e$ (i.e. an element satisfying $e^2 = e$) then $H$ is a maximal subgroup (with respect to inclusion) of $S$, with identity $e$, and conversely every maximal subgroup of $S$ arises in this way. Thus maximal subgroups and group $\gh$-classes are one and the same.

As mentioned in the introduction, Sch\"{u}tzenberger \cite{Schutzenberger1957, Schutzenberger1958} showed that in a natural way one may associate a group $\gG(H)$ with an arbitrary $\gh$-class $H$ of a monoid. This is done in such a way that if $H$ does contain an idempotent, and thus is a maximal subgroup of $S$, then $H \cong \gG(H)$, so  the notion of Sch\"{u}tzenberger group directly generalises that of maximal subgroup.

Let $S$ be a monoid, let $H$ be an $\gh$-class of $S$, and let $h \in H$ be an arbitrary fixed element of $H$. The Sch\"{u}tzenberger group of $H$ is obtained by taking the setwise stabilizer of $H$ under the right multiplicative action of $S$ on itself, and making it faithful. More precisely, let $\mathrm{Stab}(H)$ denote the right setwise stabiliser of the set $H$, so
\[
\mathrm{Stab}(H) = \{ s \in S: Hs = H  \},
\]
and define a relation $\sigma = \sigma(H)$ on $\mathrm{Stab}(H)$ by
\[
\sigma(H) = \{ (s,t) \in \mathrm{Stab}(H) \times \mathrm{Stab}(H) : hs = ht \}.
\]
It is easy to see that $\sigma$ is a congruence, which we call the \emph{Sch\"{u}tzenberger congruence} of $H$. It may then be checked that the quotient $\mathrm{Stab}(H) / \sigma$ is a group (whose isomorphism type is independent of the choice of $h \in H$), that we call the \emph{Sch\"{u}tzenberger group} of $H$, and denote by $\gG(H)$.
Of course, there is an obvious dual notion of left Sch\"{u}tzenberger group, but as it turns out the left and right Sch\"{u}tzenberger groups are naturally isomorphic to each other. For proofs of these facts, and more background on Sch\"{u}tzenberger groups, we refer the reader to \cite{Lallement1}.

\section{Proof of Theorem~\ref{thm_themainresult}}
\label{sec_up}

In this section we shall prove our main result Theorem~\ref{thm_themainresult}.
Let us begin by fixing some notation that will remain in force for the rest of the section. 
Let $A = \{a, a^{-1}, b, b^{-1} \}$ and let $R$ denote the set of rules
\[
I_x: xx^{-1} \rightarrow 1 \ (x \in A).
\]
Let $G$ denote the monoid defined by the presentation $\lb A | R \rb$. Clearly $G$ is isomorphic to the free group $F(a,b)$ over $\{a,b\}$.
Let $\lb B | Q \rb$ be the presentation with generators $B = A \cup \{ h \}$ and relations $Q = R \cup R'$ where
\[
R' =
\left\{
\begin{array}{lclcll}
K_x &:& xh & \rightarrow & hx & (x \in A) \\
C_{\epsilon,\delta} &:& ha^\epsilon b^\delta & \rightarrow & hb^\delta a^\epsilon & (\epsilon, \delta \in \{ +1, -1 \}) \\
Z_y &:& h^2 y & \rightarrow & h^2 & (y \in B).
\end{array}
\right.
\]
Here we have assigned names to the rules for easy reference. 
Let $M$ be the monoid defined by the presentation $\lb B | Q \rb$. 
The presentation $\lb B | Q \rb$ is exactly that which appears in the statement of Theorem~\ref{thm_themainresult}. 

The following result  determines a natural set of normal forms for the elements of $M$ which we then use to describe the structure of $M$, thus establishing part (i) of Theorem~\ref{thm_themainresult}.  

\begin{prop}
\label{prop_normalforms}
With the above notation, let $M$ be the monoid defined by the presentation $\lb B | Q \rb$. 
\begin{enumerate}[{\normalfont (i)}]
\item A set of normal forms for the elements of $M$ is given by
\[
\mathcal{N} =
F(a,b) \cup \{ hb^ja^k : j,k \in \mathbb{Z} \} \cup \{ h^2 \},
\]
where $F(a,b)$ denotes the set of all reduced words of the free group over $\{ a,b\}$.
\item The monoid $M$ has exactly three $\gh$-classes which, identifying $M$ with the set of normal forms $\mathcal{N}$, are
	\begin{itemize}
	\item $H_1 = F(a,b)$: a group $\gh$-class isomorphic to the free group $G$ ($H_1$ is the group of units of the monoid $M$);
	\item $H_h = \{ hb^ja^k : j,k \in \mathbb{Z} \}$: a non-group $\gh$-class with Sch\"{u}tzenberger group $\gG(H_h)$ isomorphic to the free abelian group of rank $2$;
	\item $H_0 = \{ h^2 \}$: a two-sided zero element of the monoid, forming a group $\gh$-class isomorphic to the trivial group.
	\end{itemize}
\end{enumerate}
In particular, $M$ has finitely many left and right ideals and each of the finitely many Sch\"{u}tzenberger groups of $M$ admits a presentation by a finite complete rewriting system, and so has finite derivation type.
\end{prop}
\begin{proof}
(i) We shall see below that by adding the infinitely many additional rules 
\[
\overline{C}_{w,\epsilon,\delta} : hwa^\epsilon b^\delta \rightarrow hw b^\delta a^\epsilon
\quad
(\epsilon, \delta \in \{ +1, -1 \}, \; w \in A^*)
\]
to $Q$ we obtain an infinite complete rewriting system equivalent to $Q$, from which the normal forms $\mathcal{N}$ can easily be read off as the irreducible words of the system.  

(ii) Working with the set of normal forms it is easy to check that $h^2 \mathcal{N} = \mathcal{N} h^2 = H_0$, \[hb^j a^k \mathcal{N} = \mathcal{N} h b^j a^k = H_h \cup H_0 \ (\mbox{for all} \; j,k \in \mathbb{Z}),\] and $u \mathcal{N} = \mathcal{N} u = \mathcal{N}$ $(\mbox{for all} \; u \in F(a,b))$. From this we deduce that $M$ has three $\gh$-classes $H_0$, $H_h$ and $H_1$.
The only remaining part of (ii) that may not be immediately obvious is the claim that the Sch\"{u}tzenberger group $\gG(H_h)$ is isomorphic to the free abelian group of rank $2$. To see this, observe that $\mathrm{Stab}(H_h) = F(a,b)$, and computing the Sch\"{u}tzenberger congruence $\sigma$ we see that for all $w_1, w_2 \in A^*$ we have
\[
(w_1,w_2) \in \sigma \Leftrightarrow hw_1 = hw_2,
\]
which holds if and only if one can transform $hw_1$ into $hw_2$ by applying the relations $I_x$, $K_x$ and $C_{\epsilon,\delta}$. Clearly this is equivalent to saying that $w_1$ and $w_2$ are words representing the same element of the free abelian group over $\{a,b\}$. Thus $\mathrm{Stab}(H_h)/\sigma = F(a,b)/\sigma$ is isomorphic to the free abelian group of rank $2$.
\end{proof}
Given a word $w \in B^*$ we shall use $\overline{w}$ to denote the unique word from the set of normal forms $\mathcal{N}$ which is equal to $w$ in $M$. In particular, given $w \in A^*$, $\overline{w}$ is the reduced word in the free group $F(a,b)$ equal to $w$.

Since $M$ has a zero element it follows that $M$ is of type left- and right-$\FPinfty$ by a recent observation of Kobayashi \cite{Kobayashi2009}. It follows from Proposition~\ref{prop_normalforms}, together with the fact that both free groups and free abelian groups have word problems solvable in linear time (see \cite{Wrathall1988}), that $M$ has a linear time solvable word problem. 

The rest of this section will be devoted to showing that the monoid $M$ does not have FDT, 
and hence is not presentable by a finite complete rewriting system, 
thus proving Theorem~\ref{thm_themainresult}(ii).

\subsection*{Outline of the proof of Theorem~\ref{thm_themainresult}(ii)}

Our approach to the proof of Theorem~\ref{thm_themainresult}(ii) is as follows. First we apply the general method described in Section~\ref{sec_prelims} using Lemmas~\ref{lem_criticalcircuits} and \ref{lem_infinitetofinite} to obtain an infinite homotopy base $\mathcal{C} \cup \mathcal{Z}$ (where $\mathcal{C}$ and $\mathcal{Z}$ will be defined below) for the derivation graph $\Gamma$ of $\lb B | Q \rb$. We then define in a natural way a mapping $\Phi: P(\Gamma) \rightarrow \ZM$ which sends each path of $\Gamma$ to some element of the integral monoid ring $\ZM$. Next we go on to observe that by restricting $\Phi$ to the set of paths $\mathcal{C}$ we obtain a subset $\Phi(\mathcal{C})$ of $\ZG$, and moreover, that if $M$ had FDT then $\Phi(\mathcal{C})$ would generate a finitely generated submodule of the right $\ZG$-module $\ZG$ (Lemma~\ref{lem_passtoZG}). Using the fact that $G = F(a,b)$ is a group, this in turn would imply that a certain subgroup $HN$ (where $N$ is the commutator subgroup of $G$ and $H$ is the cyclic subgroup generated by $a$) of $G$ would have to be finitely generated (Corollary~\ref{cor_HN}). But $HN$ has infinite index in $G$ (Lemma~\ref{lem_nolabel}) which, since $G$ is free and $HN$ contains the non-trivial normal subgroup $N$ of $G$, implies, by a classical result from combinatorial group theory (Theorem~\ref{thm_MKS}), that $HN$ is not finitely generated and thus $M$ does not have FDT.

We begin by finding an infinite homotopy base.

\subsection*{\boldmath Completing $Q$ to an infinite complete equivalent system $\overline{Q}$}

For each $w \in A^*$ and $\epsilon, \delta \in \{ +1, -1 \}$ define the rule
\[
\overline{C}_{w,\epsilon,\delta} : hwa^\epsilon b^\delta \rightarrow hw b^\delta a^\epsilon.
\]
Note that in particular we have $\overline{C}_{1,\epsilon,\delta} = C_{\epsilon,\delta}$. Let 
\[
\overline{Q}
=
Q \cup \{ \overline{C}_{w,\epsilon,\delta} : w \in A^*, \epsilon, \delta \in \{+1,-1\} \}.
\]
By considering the (left-to-right) length-plus-lexicographic ordering on $B^*$ induced by $a > a^{-1} > b > b^{-1} > h$ one sees that the rewriting system  $\overline{Q}$ is noetherian. Then a routine analysis of the critical peaks (the most important of which are displayed in Figure~\ref{fig_criticalcircuits}) shows that $\overline{Q}$ is an infinite complete rewriting system that is equivalent to $Q$.
Let $\Gamma$ denote the derivation graph of $\lb B | Q \rb$, and $\overline{\Gamma}$ the derivation graph of $\lb B | \overline{Q} \rb$. Let $\Gamma_Z$ denote the unique connected component of $\Gamma$ with vertex set the set of all words in $B^*$ with at least two occurrences of the letter $h$. In other words, $\Gamma_Z$ is the connected component of all words representing the zero element of the monoid $M$. Likewise let $\overline{\Gamma_Z}$ be the connected component of $\overline{\Gamma}$ with the same vertex set as $\Gamma_Z$.

\subsection*{\boldmath An infinite homotopy base $\overline{\mathcal{C}} \cup \overline{\mathcal{Z}}$ for $\overline{\Gamma}$}

The derivation graph $\overline{\Gamma}$ contains the critical circuits displayed in Figure~\ref{fig_criticalcircuits}.
Let $\overline{\mathcal{C}}$ denote the collection of all paths $\mathrm{(\overline{CT1})}$--$\mathrm{(\overline{CT7})}$ displayed in Figure~\ref{fig_criticalcircuits}, and let $\overline{\mathcal{Z}}$ denote a fixed set of critical circuits given by resolving each of the critical peaks contained in the connected component $\overline{\Gamma_Z}$.
A routine systematic check of possible overlaps of left hand sides of rules from $\overline{Q}$ reveals that $\overline{\mathcal{C}} \cup \overline{\mathcal{Z}}$ constitutes a complete set of resolutions of all possible critical peaks of the system $\overline{Q}$. Thus, by Lemma~\ref{lem_criticalcircuits},  $\overline{\mathcal{C}} \cup \overline{\mathcal{Z}}$ is a homotopy base for $\overline{\Gamma}$.

\subsection*{\boldmath An infinite homotopy base $\mathcal{C} \cup \mathcal{Z}$ for ${\Gamma}$}

The edge $\overline{C}_{w,\epsilon,\delta}$ of $\overline{\Gamma}$ is realised by the path ${C}_{w,\epsilon,\delta}$ in $\Gamma$ defined by first setting $C_{1,\epsilon,\delta} = \overline{C}_{1,\epsilon,\delta} = C_{\epsilon,\delta}$ and then defining inductively ${C}_{w,\epsilon,\delta}$ by
\begin{align}\label{eqn_C}
hxw'a^\epsilon b^\delta
\xrightarrow{K_x^{-1} \cdot w' a^\epsilon b^\delta}
xhw'a^\epsilon b^\delta
\xrightarrow{x \cdot {C}_{w',\epsilon,\delta}}
xhw'b^\delta a^\epsilon
\xrightarrow{K_x \cdot w' b^\delta a^\epsilon}
hxw'b^\delta a^\epsilon
\end{align}
where $\epsilon, \delta \in \{ +1, -1 \}$, and  $w \equiv xw'$ with $x \in A$ and $w' \in A^*$.
Let $\varphi: P(\overline{\Gamma}) \rightarrow P(\Gamma)$ be the map given by $\varphi(\alpha \cdot \overline{C}_{w,\epsilon,\delta} \cdot \beta) = \alpha \cdot C_{w,\epsilon,\delta} \cdot \beta$, for all $\alpha, \beta \in B^*$, and defined to be the identity on every other edge of $\overline{\Gamma}$.
Let $\mathcal{C} = \varphi(\overline{\mathcal{C}})$ and $\mathcal{Z} = \varphi(\overline{\mathcal{Z}})$.
Since $\overline{\mathcal{C}} \cup \overline{\mathcal{Z}}$ is a homotopy base for $\overline{\Gamma}$
it follows from Lemma~\ref{lem_infinitetofinite} that $\mathcal{C} \cup \mathcal{Z}$ is a homotopy base for $\Gamma$.

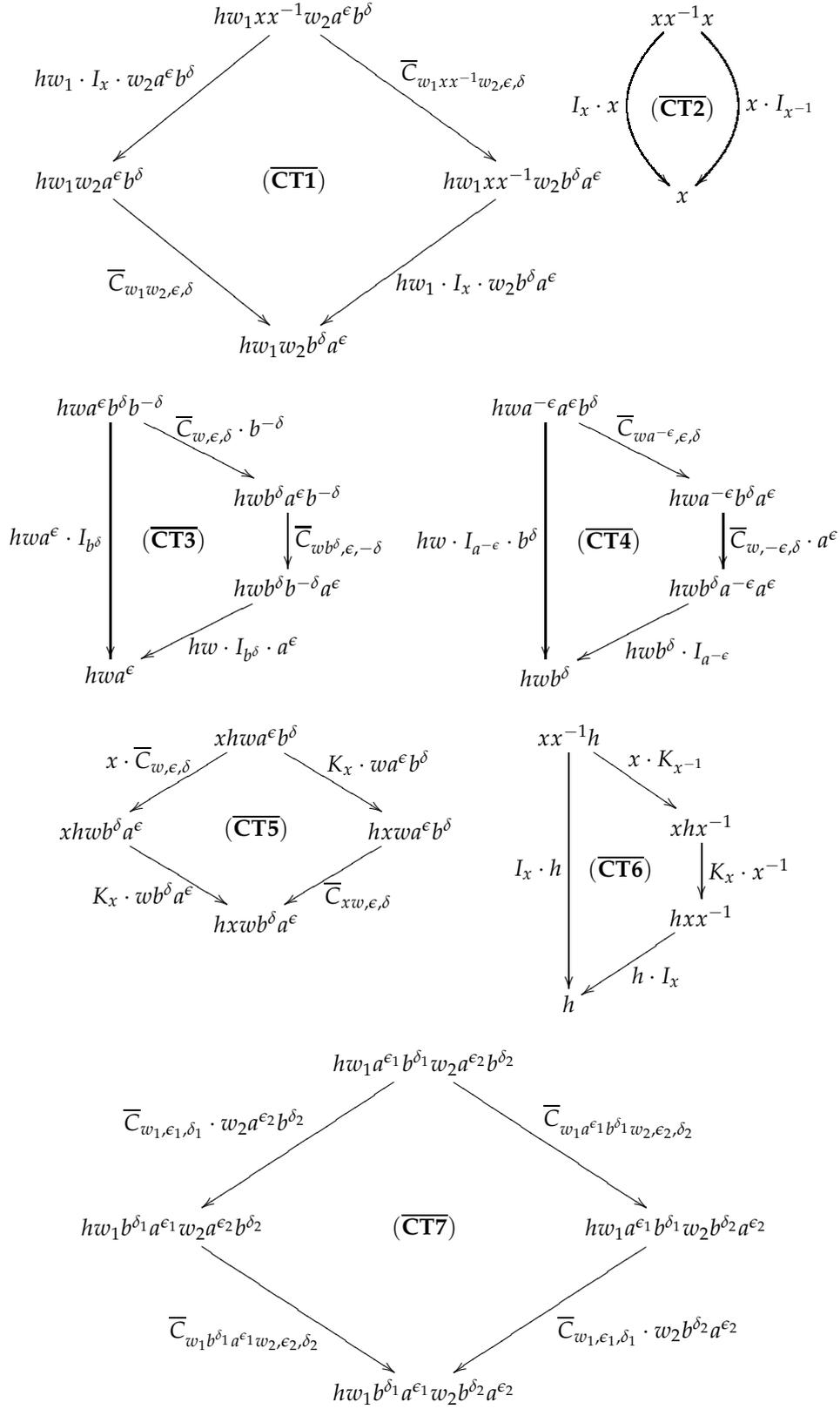
\begin{figure}
\begin{small}
$$
\begin{array}{cccc}
\xymatrix{						&		\ar[ddl]_*{hw_1 \cdot I_x \cdot w_2a^{\epsilon}b^{\delta} \  }	 	hw_1xx^{-1}w_2a^{\epsilon}b^{\delta}
\ar[ddr]^*{\  \overline{C}_{w_1xx^{-1}w_2,\epsilon,\delta}  }		&							\\
				&				&				\\
				hw_1w_2a^{\epsilon}b^{\delta} \ar[ddr]_*{\overline{C}_{w_1w_2,\epsilon,\delta}  }	&			\mathbf{(\overline{CT1})}								 &		hw_1xx^{-1}w_2b^{\delta}a^{\epsilon}   \ar[ddl]^*{\   hw_1 \cdot I_x \cdot w_2b^{\delta}a^{\epsilon}    }		\\
				&				&				\\
								&						hw_1w_2b^{\delta}a^{\epsilon}					&					}
&
&
\hspace{-14mm}
\xymatrix{	xx^{-1}x
			\ar@/_2pc/[dd]_*{I_x \cdot x}  		\ar@/^2pc/[dd]^*{x \cdot I_{x^{-1}}}
		\\      \mathbf{(\overline{CT2})}
		\\ 	x 									}
\end{array}
$$

$$
\begin{array}{cccc}
\xymatrix{	hwa^{\epsilon} b^{\delta} b^{-\delta} 	\ar[ddd]^*{\ \ \ \ \mathbf{(\overline{CT3})}}_*{hw a^{\epsilon} \cdot I_{b^{\delta}}}	 \ar[dr]^*{\overline{C}_{w, \epsilon, \delta} \cdot b^{-\delta}}			& 							\\
														&		hwb^{\delta} a^{\epsilon} b^{-\delta}	\ar[d]^*{\overline{C}_{wb^{\delta}, \epsilon, -\delta}}		\\
														&		hw b^{\delta} b^{-\delta} a^{\epsilon}	\ar[dl]^*{\ \ \ hw \cdot I_{b^\delta} \cdot a^{\epsilon}}		\\
			hwa^{\epsilon}											&							}
&
&
\hspace{-4mm}
\xymatrix{	hwa^{-\epsilon}a^{\epsilon}b^{\delta} 	\ar[ddd]^*{\ \ \ \ \mathbf{(\overline{CT4})}}_*{hw \cdot I_{a^{-\epsilon}} \cdot b^\delta}	 \ar[dr]^*{\overline{C}_{wa^{-\epsilon}, \epsilon, \delta}}			& 							\\
														&		hwa^{-\epsilon}b^{\delta}a^{\epsilon}	\ar[d]^*{\overline{C}_{w,-\epsilon,\delta} \cdot	 a^\epsilon }	\\
														&		hwb^{\delta}a^{-\epsilon}a^{\epsilon}	\ar[dl]^*{\ \ \ hwb^\delta \cdot I_{a^{-\epsilon}}}		\\
			hwb^{\delta}											&							}
\end{array}
$$

$$
\begin{array}{cccc}
\xymatrix{						&		\ar[dl]_*{x \cdot \overline{C}_{w,\epsilon,\delta}}	 	xhwa^{\epsilon} b^{\delta}	\ar[dr]^*{\ \ K_x \cdot wa^{\epsilon} b^{\delta}}		&							\\
				xhwb^{\delta}a^{\epsilon}	\ar[dr]_*{K_x \cdot wb^{\delta} a^{\epsilon}}	&				\mathbf{(\overline{CT5})}							 &		hxwa^{\epsilon}b^{\delta}	\ar[dl]^*{\overline{C}_{xw,\epsilon,\delta}}		\\
								&						hxwb^{\delta}a^{\epsilon}					&					}
&
&
\xymatrix{	xx^{-1}h 	\ar[ddd]^*{\ \ \mathbf{(\overline{CT6})}}_*{I_x \cdot h}	\ar[dr]^*{x \cdot K_{x^{-1}}}			& 							 \\
														&		xhx^{-1}	\ar[d]^*{K_x \cdot x^{-1}}		\\
														&		hxx^{-1}	\ar[dl]^*{h \cdot I_x}		\\
			h											&							}
\end{array}
$$

$$
\xymatrix{						&		\ar[ddl]_*{\overline{C}_{w_1,\epsilon_1,\delta_1}\cdot w_2a^{\epsilon_2}b^{\delta_2} \ \ \ \ }	 	 hw_1a^{\epsilon_1}b^{\delta_1}w_2a^{\epsilon_2}b^{\delta_2}	\ar[ddr]^*{\ \ \ \overline{C}_{w_1a^{\epsilon_1}b^{\delta_1}w_2,\epsilon_2,\delta_2}}		 &							\\
				&				&				\\
				hw_1b^{\delta_1}a^{\epsilon_1}w_2a^{\epsilon_2}b^{\delta_2}	 \ar[ddr]_*{\overline{C}_{w_1b^{\delta_1}a^{\epsilon_1}w_2,\epsilon_2,\delta_2}}	&					\mathbf{(\overline{CT7})}						&		 hw_1a^{\epsilon_1}b^{\delta_1}w_2b^{\delta_2}a^{\epsilon_2}	\ar[ddl]^*{\ \ \ \ \ \ \ \overline{C}_{w_1,\epsilon_1,\delta_1} \cdot w_2b^{\delta_2}a^{\epsilon_2}}		\\
				&				&				\\
								&						hw_1b^{\delta_1}a^{\epsilon_1}w_2b^{\delta_2}a^{\epsilon_2}					&					 }
$$
\end{small}
\caption{A set $\overline{\mathcal{C}} = \{ \mathrm{(\overline{CT1})}$--$\mathrm{(\overline{CT7})} \} $ of critical circuits in $\overline{\Gamma}$ given by
resolving critical peaks. Here $x \in A$, $w, w_1, w_2 \in A^*$ and $\epsilon, \epsilon_1, \epsilon_2, \delta, \delta_1, \delta_2 \in \{ +1, -1 \}$.
A corresponding set $\mathcal{C} = \{ \mathrm{({CT1})}$--$\mathrm{({CT7})} \} $ of closed paths in $\Gamma$ is obtained by replacing each occurrence of an edge of the form $\overline{C}_{w,\epsilon,\delta}$ by the path ${C}_{w,\epsilon,\delta}$ defined in \eqref{eqn_C}.
}
\label{fig_criticalcircuits}
\end{figure}

\subsection*{\boldmath Mapping into the integral monoid ring $\ZM$} Now define $\Phi: P(\Gamma) \rightarrow \ZM$ to be the unique map which extends:
\begin{itemize}
\item $\Phi(\alpha \cdot K_a \cdot \beta) = \overline{\beta}$;
\item $\Phi(\alpha \cdot K_{a^{-1}} \cdot \beta) = -\overline{\beta}$;
\item $\Phi(\alpha \cdot E \cdot \beta) = 0$ for every rewrite rule $E \in Q$ with $E \neq K_a, K_{a^{-1}}$,
\end{itemize}
to paths  in such a way that
\[
\Phi(\bbp \circ \bbq) = \Phi(\bbp) + \Phi(\bbq) 
\;\; \mbox{and} \;\;
\Phi(\bbp^{-1}) = - \Phi(\bbp).
\]
The following basic properties of $\Phi$ are then easily verified for all paths $\bbp, \bbq \in P(\Gamma)$ and words $\alpha,\beta \in B^*$:
\begin{enumerate}[(i)]
\item $\Phi(\alpha \cdot \bbp \cdot \beta) = \Phi(\bbp) \cdot \overline{\beta}$
\item $\Phi(\bbp \circ \bbp^{-1}) = 0$
\item $\Phi([\bbp,\bbq]) = 0$ where
\[
[\bbp,\bbq] =
(\bbp \cdot \iota \bbq)
\circ
(\tau \bbp \cdot \bbq)
\circ
(\bbp^{-1} \cdot \tau \bbq)
\circ
(\iota \bbp \cdot \bbq^{-1}).
\]
\item If $\bbp \sim_0 \bbq$ then $\Phi(\bbp) = \Phi(\bbq)$.
\end{enumerate}
Here, (iv) follows from (ii) and (iii). Note that (iv) implies that $\Phi$ induces a well-defined map on the homotopy classes of paths of $\Gamma$.

In what follows we shall often omit bars from the top of words in the images under $\Phi$ and simply write words from $B^*$ with the obvious intended meaning.

\subsection*{\boldmath Computing images $\Phi(C)$ for $C \in \mathcal{C}$}
Now consider the effect of applying the mapping $\Phi$ to the closed paths from $\mathcal{C}$, where we take the convention that each path is read clockwise. Define a mapping $\partial:A^* \rightarrow \ZG$ by setting $\partial w = 0$ when $w = 1$, 
\[
\partial w =
\begin{cases}
-1 & \mbox{if} \ w=a \\
1 & \mbox{if} \ w=a^{-1} \\
0 & \mbox{if} \ w \in \{ b, b^{-1} \},
\end{cases}
\]
and when $|w|>1$ define inductively 
\[
\partial w = (\partial x) w' + \partial w'
\]
where $w \equiv xw'$ with $x \in A$ and $w' \in A^+$. 
Note that for all $x \in A$ and $\epsilon = \pm 1$ we have $\partial x^{-1} = - \partial x$ and $\partial a^{\epsilon} = -\epsilon$.
Using the map $\partial$ we may readily deduce the following equations 
\begin{enumerate}
\item[(i)] 
$\Phi(K_x) = - \partial x $ 
\item[(ii)] 
$\Phi(C_{w,\epsilon,\delta}) = -(\partial w) (b^\delta a^\epsilon - a^\epsilon b^\delta)$
\item[(iii)]
$\Phi(C_{xw,\epsilon,\delta}) = \Phi(C_{w,\epsilon,\delta}) - (\partial x) w (b^\delta a^\epsilon - a^\epsilon b^\delta)$
\item[(iv)] 
$\Phi(C_{w_1 w_2,\epsilon,\delta}) = \Phi(C_{w_2,\epsilon,\delta}) - (\partial w_1) w_2 (b^\delta a^\epsilon - a^\epsilon b^\delta)$
\end{enumerate}
for any $x \in A$, $w, w_1, w_2 \in A^*$ and $\epsilon, \delta \in \{ -1, +1 \}$. Routine calculations using these equations then yield the
results of applying $\Phi$ to each of the critical circuits from $\mathcal{C}$. The results of these computations are given in the table in Figure~\ref{fig_table}.
\begin{figure}
\renewcommand{\tabcolsep}{6pt}
\begin{tabular}{ |  c | l | }
\hline
Circuit type of $\bbp$ & \quad \quad \quad  \quad \quad \quad  \quad \quad  $\Phi(\bbp)$  \
\\ \hline \hline
  (CT1)  &
	$
	\begin{cases}
	e(a^{-e} - 1)w_2(b^\delta a^\epsilon - a^\epsilon b^\delta) &  \mbox{if} \  x = a^e \\
  0 & \mbox{if} \  x \in \{ b,b^{-1} \} 		
	\end{cases}
	$
	\\ \hline 
  (CT2)  & 	
	$
	0
	$
	\\ 
  \hline
  (CT3)  & 	
	$
  0
	$
	\\ 
  \hline
  (CT4)  & 	
	$
	-\epsilon (b^\delta a^\epsilon - a^\epsilon b^\delta)	
	$
	\\ 
  \hline
  (CT5)  & 	
	$
  0
	$
	\\ 
  \hline
  (CT6)  &
	$
	\begin{cases}
	e (a^{-e} - 1) &  \mbox{if} \  x = a^e \\
	0 & \mbox{if} \  x \in \{ b,b^{-1} \}
	\end{cases}
	$
	\\ 
  \hline
  (CT7)  &
  $\epsilon_1 (b^{\delta_1} - 1)w_2(b^{\delta_2} a^{\epsilon_2} - a^{\epsilon_2} b^{\delta_2})$
	\\ \hline 
\end{tabular}
\caption{The images under $\Phi$ of the critical circuits from $\mathcal{C}$.}
\label{fig_table}
\end{figure}
Observe that $\Phi(\mathcal{C})$ is a subset of $\ZG$ where $G$ is the free group $F(a,b)$ over $\{ a,b \}$.

For a subset $X$ of a right $\ZG$-module $\ZG$ we use $\lb X \rb_{\ZG}$ to denote the submodule generated by $X$. 
\begin{lem}
\label{lem_passtoZG}
If $M$ has $\mathrm{FDT}$ then the submodule $\lb \Phi(\mathcal{C}) \rb_\ZG$, of the right $\ZG$-module $\ZG$,
generated by $\Phi(\mathcal{C})$
is a finitely generated right $\ZG$-module.
\end{lem}
\begin{proof}
Since $\lb B | Q \rb$ has $\mathrm{FDT}$, and $\mathcal{C} \cup \mathcal{Z}$ is a homotopy base for its derivation graph $\Gamma$, it follows that there are finite subsets $\mathcal{C}_0 \subseteq \mathcal{C}$ and $\mathcal{Z}_0 \subseteq \mathcal{Z}$ such that $\sim_{\mathcal{C}_0 \cup \mathcal{Z}_0}$ is a finite homotopy base for $\Gamma$. Let $C \in \mathcal{C}$ be arbitrary. We claim that $\Phi(C) \in \lb \Phi(\mathcal{C}_0) \rb_\ZG$. Once established, this will prove the lemma, since $\Phi(\mathcal{C}_0)$ is a finite subset of $\lb \Phi(\mathcal{C}) \rb_\ZG$.

By Lemma~\ref{lem_homgen}, since $C$ is a closed path in $\Gamma$ and 
$\sim_{\mathcal{C}_0 \cup \mathcal{Z}_0}$ is a homotopy base for $\Gamma$,
we can write
\begin{align}
\label{eqn_pathdecomp}
C
\sim_0
\bbp_1^{-1} \circ (\alpha_1 \cdot \bbq_1 \cdot \beta_1) \circ \bbp_1
\circ \cdots \circ
\bbp_n^{-1} \circ (\alpha_n \cdot  \bbq_n \cdot  \beta_n) \circ \bbp_n,
\end{align}
where each $\bbp_i \in P(\Gamma)$, $\alpha_i, \beta_i \in B^*$ and $\bbq_i \in \mathcal{C}_0 \cup \mathcal{Z}_0$. Since $C \in \mathcal{C}$ it follows that $C$ is a closed path in $\Gamma$ contained in some connected component $\mathcal{D}$ of $\Gamma$ that is disjoint from $\Gamma_Z$ (since $\iota C$ does not contain more than one letter $h$). Therefore, since the path on the right hand side of \eqref{eqn_pathdecomp} is $\sim_0$-homotopic to $C$ in $\Gamma$ it follows that this path too is contained in the connected component $\mathcal{D}$ of $\Gamma$ where $\mathcal{D} \neq \Gamma_Z$. In particular, this implies that
for $1 \leq j \leq n$ the word $\iota (\alpha_j \cdot  \bbq_j \cdot  \beta_j)$ has at most one occurrence of the letter $h$, and therefore $\bbq_j \in \mathcal{C}_0$ and moreover by inspection of the circuits (CT1)--(CT7) we see that, whenever $\bbq_j$ is not of the form (CT2), we must have $\beta_j \in A^*$ (since otherwise the word $\iota (\alpha_j \cdot  \bbq_j \cdot  \beta_j)$ would have strictly more than one occurrence of the letter $h$).

Applying $\Phi$ to \eqref{eqn_pathdecomp} then gives
\begin{align}
\label{eqn_PHIpathdecomp}
\Phi(C)
=
\Phi(\bbq_1)\beta_1 + \ldots + \Phi(\bbq_n)\beta_n.
\end{align}
Now for $1 \leq j \leq n$, if $\bbq_j$ has any of the forms (CT2), (CT3) or (CT5) then $\Phi(\bbq_j)=0$. Hence the non-zero terms in the sum \eqref{eqn_PHIpathdecomp} are made up entirely of images of paths from (CT1), (CT4), (CT6) and (CT7). Therefore, from the observation in the previous paragraph, it follows that whenever $\Phi(\bbq_j) \neq 0$ we have $\beta_j \in A^*$. Along with the fact that every $\bbq_j \in \mathcal{C}_0$ this shows
\begin{align*}
\Phi(C)
=
\Phi(\bbq_1)\beta_1 + \ldots + \Phi(\bbq_n)\beta_n \in \lb \Phi(\mathcal{C}_0) \rb_\ZG,
\end{align*}
as claimed, completing the proof of the lemma.
\end{proof}

The next lemma describes the submodule of the right $\ZG$-module $\ZG$ generated by $\Phi(\mathcal{C})$.

\begin{lem}
\label{lem_equivalent}
Let
\[
X = \{ (1-a) \} \cup
\{
(1 - w a^\epsilon b^\delta a^{-\epsilon} b^{-\delta} w^{-1}):
w \in G
\}  \subseteq \ZG.
\]
Then
\[
\lb \Phi(\mathcal{C}) \rb_\ZG = \lb X \rb_\ZG.
\]
\end{lem}
\begin{proof}
First we show $X \subseteq \lb \Phi(\mathcal{C}) \rb_\ZG$. That $(1-a) \in \lb \Phi(\mathcal{C}) \rb_\ZG$ follows immediately from consideration of the images of the paths (CT6) under $\Phi$. Next we prove by induction on the length of the reduced word $\overline{w}$ in the free group $G$, that for all $\epsilon, \delta \in \{+1,-1\}$ we have
\[
(1- wa^\epsilon b^\delta a^{-\epsilon} b^{-\delta} w^{-1}) \in \lb \Phi(\mathcal{C}) \rb_\ZG.
\]
The base case $w=1 \in G$ follows by consideration of the images of the paths (CT4) under $\Phi$.
For the induction step, suppose that $w$ is in reduced form. First suppose that $w \equiv a^e w'$ where $e \in \{+1,-1\}$. Then considering the $\Phi$-images of paths (CT1) we see that
\[
ea^{-e}w(b^\delta a^\epsilon - a^\epsilon b^\delta)
-ew(b^\delta a^\epsilon - a^\epsilon b^\delta)
=
e(a^{-e} - 1)w(b^\delta a^\epsilon - a^\epsilon b^\delta) 
\in \lb \Phi(\mathcal{C}) \rb_\ZG.
\]
But by induction
\[
ea^{-e}w(b^\delta a^\epsilon - a^\epsilon b^\delta)
=
e w' (b^\delta a^\epsilon - a^\epsilon b^\delta) \in \lb \Phi(\mathcal{C}) \rb_\ZG
\]
and so it follows that
\[
-ew(b^\delta a^\epsilon - a^\epsilon b^\delta) \in \lb \Phi(\mathcal{C}) \rb_\ZG,
\]
completing the induction step in this case. The other possibility is that $w \equiv b^e w'$ where $e \in \{+1,-1\}$. The argument in this case is similar. Considering $\Phi$-images of paths (CT7) we see that
\[
b^ew' (b^\delta a^\epsilon - a^\epsilon b^\delta)
-w' (b^\delta a^\epsilon - a^\epsilon b^\delta)
=
(b^e-1)w' (b^\delta a^\epsilon - a^\epsilon b^\delta) 
\in \lb \Phi(\mathcal{C}) \rb_\ZG
\]
and the result then follows since
\[
w'(b^\delta a^\epsilon - a^\epsilon b^\delta) \in \lb \Phi(\mathcal{C}) \rb_\ZG
\]
by induction, and thus
\[
w(b^\delta a^\epsilon - a^\epsilon b^\delta) = b^e w'(b^\delta a^\epsilon - a^\epsilon b^\delta)
\in \lb \Phi(\mathcal{C}) \rb_\ZG.
\]
Conversely, the fact that $\Phi(\mathcal{C}) \subseteq \lb X \rb_\ZG$ follows easily from inspection of the images of (CT1)--(CT7) under $\Phi$, listed in the table in Figure~\ref{fig_table}.  \end{proof}

Therefore to complete the proof of Theorem~\ref{thm_themainresult} it will suffice to show that $\lb X \rb_\ZG$ is \emph{not} finitely generated as a right $\ZG$-module. For this we make use of the following general result. 
\begin{lem}
\label{lem_wellknown}
Let $G$ be a group, let $A$ be a subset of $G$ and let
\[
(1-A) = \{ (1-a) : a \in A  \} \subseteq \ZG.
\]
Then for all $g \in G$, if $(1-g) \in \lb (1-A) \rb_\ZG$ then $g \in \lb A \rb$ in $G$.
In particular, if $\lb (1-A) \rb_\ZG$ is a finitely generated right $\ZG$-module then $\lb A \rb$ is a finitely generated group.
\end{lem}
\begin{proof}
This result is almost certainly well known; see for instance \cite[Section~3, Exercise~2]{Brown}. We include a proof here for the sake of completeness.

Let $H = \lb A \rb$ be the subgroup of $G$ generated by $A$. Let $X = \{ Hg : g \in G \}$ be the set of right cosets of $H$ in $G$. Of course, $G$ acts on $X$ on the right via
\[
(H g_1) \cdot g_2 = H g_1 g_2.
\]
Let $\mathbb{Z}X = \oplus_{x \in X} \mathbb{Z} x$ denote the free abelian group with basis $X$. Define an action of $\ZG$ on $\mathbb{Z}X$ by
\[
(x_1 + \cdots + x_k) \cdot
(g_1 + \cdots + g_r)
=
\sum_{1 \leq i \leq k \atop 1 \leq j \leq r} x_i \cdot g_j \in \mathbb{Z}X.
\]
It is easy to see that with respect to this action
$\mathbb{Z}X$ is a right $\ZG$-module. 
Let $x_H = H1 \in X$.
 Now let $g \in G$ with $(1-g) \in \lb (1-A) \rb_\ZG$.
 This means we can write
 \begin{equation}\label{eq_star}
 (1-a_1) \lambda_1 + \cdots + (1-a_t) \lambda_t = 1-g
 \end{equation}
 where each $a_i \in A$ and $\lambda_i \in \ZG$. But for every $h \in H$ we have
 \[
 x_H \cdot (1-h) = x_H - x_H = 0.
 \]
 Since $A \subseteq H = \lb A \rb$, from \eqref{eq_star} we conclude
 \[
 x_H \cdot (1-g) =
 x_H \cdot ( \; (1-a_1) \lambda_1 + \cdots + (1-a_t) \lambda_t \;) =
 0 + \cdots + 0 = 0.
 \]
 It follows that $x_H \cdot g = x_H$ so $Hg = H$ which implies $g \in H = \lb A \rb$.

 For the last clause, if $\lb (1-A) \rb_\ZG$ is a finitely generated right $\ZG$-module then there is a finite subset $A'$ of $A$ such that $(1-A) \subseteq \lb (1-A') \rb_\ZG$ which in turn from above implies that $A \subseteq \lb A' \rb$ and so $A'$ is a finite generating set for $\lb A \rb$.
\end{proof}
\begin{cor}
\label{cor_HN}
Let $H = \lb a \rb$, the cyclic subgroup of the free group $G = F(a,b)$ generated by $a$, and let
\[
N =
\lb \
[b^\delta, a^\epsilon]^w : w \in G, \epsilon, \delta \in \{ +1, -1 \}
\ \rb,
\]
where $[x,y]$ denotes the commutator $xyx^{-1}y^{-1}$, and $x^y = yxy^{-1}$. If $M$ has $\mathrm{FDT}$ then the subgroup $\lb H \cup N \rb = HN \leq G$ is finitely generated.
\end{cor}
\begin{proof}
Suppose that $M$ has $\mathrm{FDT}$. Let
\[
X = \{ (1-a) \} \cup
\{
(1 - w a^\epsilon b^\delta a^{-\epsilon} b^{-\delta} w^{-1}):
w \in G
\}  \subseteq \ZG.
\]
Since $M$ has $\mathrm{FDT}$, by Lemmas~\ref{lem_passtoZG} and \ref{lem_equivalent} it follows that $\lb X \rb_\ZG$ is finitely generated as a right $\ZG$-module. It then follows from the last clause of Lemma~\ref{lem_wellknown} that
\[
HN = \lb H \cup N \rb =
\lb
\{ a \} \cup
\{ w a^\epsilon b^\delta a^{-\epsilon} b^{-\delta} w^{-1} : w \in G \}
\rb
\]
is a finitely generated group.
\end{proof}

Of course $N$ is nothing more than the commutator subgroup of the free group $G=F(a,b)$.

To complete the proof of Theorem~\ref{thm_themainresult} we
apply the following classical result from combinatorial group theory.
\begin{thm}[\cite{Magnus1}, Theorem~2.10]
\label{thm_MKS}
Let $F$ be a non-abelian free group of finite rank and let $K$ be a subgroup of $F$ with index $i$. If $i$ is infinite and $K$ contains a normal subgroup $L$ of $F$, with $L \neq 1$, then $K$ is \emph{not} finitely generated.
\end{thm}

\begin{lem}
\label{lem_nolabel}
The subgroup $HN$ has infinite index in the free group $G$, and therefore $HN$ is not finitely generated.
\end{lem}
\begin{proof}
First observe that $b^k \not\in HN$ for all $k \neq 0$. Indeed, for any word $w \in A^*$, if $w$ represents an element of $HN$ then the sum of the exponents of the $b$'s of $w$ must equal zero. It follows that for all $k,l \in \mathbb{N}$ if $k \neq l$ then $b^k$ and $b^l$ belong to distinct cosets of $HN$. Therefore $HN$ has infinite index in $G$. The last statement is then a consequence of Theorem~\ref{thm_MKS}.  
\end{proof}

Since $HN$ is not finitely generated it follows by Corollary~\ref{cor_HN} that $M$ does not have finite derivation type. This completes the proof of Theorem~\ref{thm_themainresult}.

\section{Applications}
\label{sec_consequences}

In this section we give two further applications of Theorem~\ref{thm_themainresult}.

\subsection*{Quasi-isometry invariance}
\begin{sloppypar}
A key concept in geometric group theory
is that of quasi-isometry: a notion of equivalence between metric spaces
which captures formally the intuitive idea of two spaces looking the same
"when viewed from far away"; see \cite{delaHarpeBook}.
The Cayley graph of a finitely generated group may
naturally be viewed as a metric space with respect to the word metric. Many important properties of finitely generated groups are then known to be shared between groups that are quasi-isometric to each other, meaning that they have Cayley graphs that are quasi-isometric as metric spaces; see \cite[p115, Section 50]{delaHarpeBook} for a list of such properties. In particular, the homological finiteness property $\mathrm{FP}_n$ is know to be a quasi-isometry invariant of finitely generated groups; see \cite{Alonso1994}. Combining this with \cite{Cremanns1996} it follows that $\mathrm{FDT}$ is a quasi-isometry invariant of finitely generated groups.
\end{sloppypar}

In a monoid Cayley graph, by contrast,
distance is neither symmetric (since there are no inverses) nor everywhere
defined (since there may be ideals). Hence, there is no hope that a general monoid will `resemble' a metric space.
Thus, rather than a metric space, a more natural geometric object to associate to a finitely generated monoid is a, so-called \emph{semimetric space} which is a set equipped with an assymetric, partially defined distance function.
This is the viewpoint taken in \cite{Gray2009} where, among other things, a natural notion of quasi-isometry for such spaces is exhibited, and several quasi-isometry invariants of monoids are identified.
The axioms for a semimetric space are given by taking the usual metric space axioms, relaxing the condition that distances are always defined (i.e. allowing points to be at distance $\infty$), and dropping the symmetry assumption so that $d(x,y)$ and $d(y,x)$ need not be equal; see \cite{Gray2009} for a formal definition. Then given a monoid $S$ and a finite generating set $A$ for $S$, we associate a semimetric space $(S,d_A)$ where $d_A$ is the obvious directed distance semimetric given by
\[
d_A(x,y) = \inf \{ |w| : w \in A^* : xw = y \}.
\]
We shall now see that, in contrast to the situation for groups, the property $\mathrm{FDT}$ is \emph{not} a quasi-isometry invariant of monoids. In fact, we do more than this. We shall actually show that $\mathrm{FDT}$ is not even an isometry invariant of finitely generated monoids. Just as for metric spaces, by an isometry of semimetric spaces we mean a distance-preserving map between semimetric spaces; see \cite[Definition~2]{Gray2009}.
\begin{thm}
\label{thm_qsi}
Let $M$ be the monoid defined by the following presentation
\begin{align*}
\lb \; a, a^{-1}, b, b^{-1}, h, z
& \; | \; aa^{-1} = a^{-1}a=bb^{-1}=b^{-1}b=1, \\
& \phantom{\; | \;}
xh = hx, \; hxy = hyx \quad (x,y \in \{ a, a^{-1}, b, b^{-1} \}),  \\
& \phantom{\; | \;}
h^2 a = h^2 a^{-1} = h^2 b = h^2 b^{-1} = h^3 = h^2, \\
& \phantom{\; | \;}
h^2=z \; \rb,
\end{align*}
and let $N$ be the monoid defined by
\begin{align*}
\lb \; a, a^{-1}, b, b^{-1}, h, z
& \; | \; aa^{-1} = a^{-1}a=bb^{-1}=b^{-1}b=1, \\
& \phantom{\; | \;}
xh = hx, \; hxy = hyx \quad (x,y \in \{ a, a^{-1}, b, b^{-1} \}),  \\
& \phantom{\; | \;}
zu = uz = z \quad (u \in \{ a, a^{-1}, b, b^{-1}, h, z \}) \\
& \phantom{\; | \;}
h^2=h \; \rb.
\end{align*}
Then, with $A = \{ a, a^{-1}, b, b^{-1}, h, z \}$, $(M,d_A)$ and $(N,d_A)$ are isometric.
However, $M$ does not have $\mathrm{FDT}$ and is not presentable by a finite complete rewriting system, while $N$ is presentable by a finite complete rewriting system, and $N$ does have $\mathrm{FDT}$.
\end{thm}
\begin{proof}
Observe that $M$ is isomorphic to the monoid defined in Theorem~\ref{thm_themainresult}, the presentation being obtained from the presentation in Theorem~\ref{thm_themainresult} by adding a redundant generator $z$ and the relation $h^2 = z$.

We must define a distance preserving bijection $f:M \rightarrow N$. A set of normal forms for $M$ is easily seen to be given by $\mathcal{M} = \mathcal{M}_1 \cup \mathcal{M}_h \cup \mathcal{M}_0$ where $\mathcal{M}_1 = F(a,b)$ denotes the set of reduced words in the free group over $\{a,b\}$, $\mathcal{M}_h = \{ hb^ja^k : j,k \in \mathbb{Z} \}$ and $\mathcal{M}_0 = \{ z \}$. Also, a set of normal forms for $N$ is easily seen to be given by $\mathcal{N} = \mathcal{N}_1 \cup \mathcal{N}_h \cup \mathcal{N}_0$ where $\mathcal{N}_1 = \mathcal{M}_1$, $\mathcal{N}_h= \mathcal{M}_h$ and $\mathcal{N}_0 = \mathcal{M}_0$. Since the sets of normal forms $\mathcal{M}$ and $\mathcal{N}$ are identical, identifying $M$ and $N$ with their sets of normal forms, we can set $f: \mathcal{M} \rightarrow \mathcal{N}$ to be the identity mapping, which of course is a bijection.
It is then a routine matter to show that $f$ is distance preserving.
Indeed, in the (right) Cayley graph of $M$, for every element in $\mathcal{M}_h$ the arc from this vertex labelled by $h$ goes to $z$. By removing all of these arcs and adding in a loop labelled by $h$ for each vertex in $\mathcal{M}_h$ we would obtain precisely the Cayley graph of $N$. 
 So in the Cayley graph of $M$ for every element in $\mathcal{M}_h$ there are two directed arcs from the element to $z$, one labelled $h$ and the other labelled $z$, while in the Cayley graph of $N$ there is only one such arc, labelled by $z$. Also in $\mathcal{N}_h$ every element has a loop labelled by $h$ (since $h$ is the identity of the group $\gh$-class) while in $\mathcal{M}_h$ no such loops exist. Since these are the only differences between the two Cayley graphs, it follows immediately that $f$ is an isometry of semimetric spaces.

The facts that $M$ does not have $\mathrm{FDT}$ and is not presentable by a finite complete rewriting system
follow from Theorem~\ref{thm_themainresult}.

On the other hand, $N$ is a regular monoid with finitely many $\gh$-classes, and all maximal subgroups of $N$ are definable by finite complete rewriting systems. Indeed, the maximal subgroups in question are the trivial group, the free abelian group of rank two, and the free group of rank two. It therefore follows from the main result of \cite{GrayMalheiro2} (see also \cite{GrayMalheiro}) that $N$ is presentable by a finite complete rewriting system, and therefore $N$ also has $\mathrm{FDT}$.  \end{proof}

\subsection*{Green index extensions} We say that $T$ is a \emph{large} subsemigroup of a semigroup $S$, and $S$ is a \emph{small extension} of $T$, if $T \leq S$ and $|S \setminus T| < \infty$. This notion was first investigated by Jura in \cite{Jura1,Jura2}. In \cite{Ruskuc1} it was shown that the property of being finitely presented is inherited by small extensions, and then in \cite{Wang1998} Wang showed that both the property of being definable by a finite complete rewriting system, and also the property FDT, are also inherited when taking small extensions. Analogously, in group theory, both of these finiteness properties are known to be preserved when taking finite index extensions (in the usual group-theoretic sense); see \cite[Proposition 5.1]{Brown} and \cite{Groves1993}. While analogous, these results are independent in the sense that the results for groups cannot (of course) be used to deduce the results regarding small extensions of semigroups, and conversely the small extensions results cannot be used to deduce the result for finite index extensions of groups. 

It is natural to ask whether the small extensions results can be generalised (by weakening the condition that the complement $S \setminus T$ is finite) in such a way as to obtain a single result that has both the semigroup and group-theoretic results as corollaries. This kind of question was one of the motivations for the work in \cite{Gray2008} where a less restrictive notion of index for semigroups was introduced, called Green index, which we now briefly describe.

Let $S$ be a semigroup and let $T$ be a subsemigroup of $S$.
We use $S^1$ to denote the semigroup $S$ with an identity element $1 \not\in S$ adjoined to it. This notation will be extended to subsets of $S$, i.e. $X^1 = X \cup \{ 1 \}$. For $u,v \in S$ define
\[
u \gr^T v  \ \Leftrightarrow  \ uT^1 = vT^1, \quad u \gl^T v \  \Leftrightarrow \  T^1u = T^1v,
\]
and $\gh^T = \gr^T \cap \gl^T$. Each of these relations is an equivalence relation on $S$; their equivalence classes are called the
($T$-)\emph{relative}  $\gr$-, $\gl$-, and $\gh$-classes, respectively.
Furthermore, these relations respect $T$, in the sense that each
$\gr^T$-, $\gl^T$-, and $\gh^T$-class lies either wholly
in $T$ or wholly in $S \setminus T$. Relative Green's relations were introduced by Wallace in \cite{Wallace} generalising the fundamental work of Green \cite{Green1}. Following \cite{Gray2008} we define the \emph{Green index} of $T$ in $S$ to be one more than the number of $\gh^T$-classes in $S \setminus T$. Clearly if $T$ is a large subsemigroup of a semigroup $S$ then $T$ has finite Green index in $S$, and also if $H$ is a finite index (in the usual group-theoretic sense) subgroup of a group $G$ then $H$ has finite Green index in $G$. 

\begin{sloppypar}
With each $T$-relative $\gh$-class we may associate a group, which we call the $T$-relative \emph{Sch\"{u}tzenberger group} of the $\gh$-class.
This is done by extending, in the obvious way, the classical definition (defined above) to the relative case.
For each $T$-relative $\gh$-class $H$ let
$\mathrm{Stab}(H) = \{ t \in T^1 : Ht = H \}$
(the \emph{stabilizer} of $H$ in $T$), and
define an equivalence $\gamma=\gamma(H)$ on $\mathrm{Stab}(H)$ by $(x,y) \in \gamma$ if and only if $hx = hy$
for all $h\in H$. Then $\gamma$ is a congruence on $\mathrm{Stab}(H)$ and $\mathrm{Stab}(H) / \gamma$ is a group.
The group $\Gamma(H)=\mathrm{Stab}(H) / \gamma$ is called the \emph{relative Sch\"{u}tzenberger group} of $H$.
\end{sloppypar}

In \cite{Cain2009}, providing a common generalisation of \cite[Theorem~4.1]{Ruskuc1} and the corresponding classical result for finite index extensions in group theory, it was proved that if $T$ is a finite Green index subsemigroup of a semigroup $S$, then if $T$ is finitely presented and all of the $T$-relative Sch\"{u}tzenberger groups of $S \setminus T$ are finitely presented, then $S$ itself is finitely presented.

Now, as pointed out above, it is known that in group theory both finite derivation type, and the property of being presentable by a finite complete rewriting system, are preserved by finite index group extensions. In \cite{Wang1998} the analogous results were proved for passing to small extensions of semigroups. Thus a natural question is whether these results have a common generalisation to finite Green index extensions, as was known to be the case for finite presentability in \cite{Cain2009}. Using the example from Section~\ref{sec_up} we now answer this question in the negative.

\begin{thm}
There exists a semigroup $S$ with a subsemigroup $T$ of finite Green index, such that:
\begin{enumerate}[{\normalfont (i)}]
\item
$T$, and all of the relative Sch\"{u}tzenberger groups of $S \setminus T$ admit presentations by finite complete rewriting system, and thus all have finite derivation type;
\item
$S$ does not have finite derivation type, and so does not admit a presentation by a finite complete rewriting system.
\end{enumerate}
\end{thm}
\begin{proof}
Let $S=M$ the monoid defined in Theorem~\ref{thm_themainresult} and let $T$ be the group of units $H_1$ of $S$. Then it follows from the analysis in the proof of Proposition~\ref{prop_normalforms} that $H_1$ has finite Green index in $S$, that the relative Sch\"{u}tzenberger groups of $S \setminus T$ are isomorphic to the trivial group and the free abelian group of rank two, respectively, and hence they admit presentations by finite complete rewriting systems, and thus all have finite derivation type. Also, $T=H_1$ which is isomorphic to the free group of rank two, and so admits a presentation by a finite complete rewriting system, and has FDT. This proves (i). Part (ii) follows from Theorem~\ref{thm_themainresult}.
\end{proof}

\bibliographystyle{abbrv}
\def\cprime{$'$} \def\cprime{$'$} \def\cprime{$'$}

\end{document}